\documentclass[reqno]{amsart}
\usepackage{amssymb}
\usepackage{mathrsfs}
\usepackage{cite}
\usepackage{graphicx}
\usepackage{tikz-cd}

\newtheorem{theorem}{Theorem}
\newtheorem{lemma}{Lemma}
\newtheorem{corollary}{Corollary}

\theoremstyle{definition}

\theoremstyle{remark}
\newtheorem{remark}{Remark}

\newcommand{\dontprint}[1]\relax

\newcommand{\sn}{\operatorname{sn}}
\newcommand{\cn}{\operatorname{cn}}
\newcommand{\dn}{\operatorname{dn}}

\newcommand{\I}{\mathrm{i}}
 
\begin{document}

\title[Poncelet property of the Boltzmann system]
{Poncelet property and quasi-periodicity
  of the integrable Boltzmann system}
\thanks{
The author is supported in part by the National Centre of Competence
in Research SwissMAP---The Mathematics of Physics---of the Swiss
National Science Foundation.
}

\author{Giovanni Felder}
\address{Department of Mathematics, ETH Zurich, 8092 Zurich, Switzerland}
\email{felder@math.ethz.ch}

\begin{abstract} We study the motion of a particle in a plane subject
  to an attractive central force with inverse-square law on one side
  of a wall at which it is reflected elastically. This model is a
  special case of a class of systems considered by Boltzmann which was
  recently shown by Gallavotti and Jauslin to admit a second integral
  of motion additionally to the energy. By recording the subsequent
  positions and momenta of the particle as it hits the wall we obtain
  a three dimensional discrete-time dynamical system. We show that
  this system has the Poncelet property: if for given generic values
  of the integrals one orbit is periodic then all orbits for these
  values are periodic and have the same period. The reason for this is
  the same as in the case of the Poncelet theorem: the generic level
  set of the integrals of motion is an elliptic curve, the
  Poincar\'e map is the composition of two involutions with fixed
  points and is thus the translation by a fixed element. Another
  consequence of our result is the proof of a conjecture of Gallavotti
  and Jauslin on the quasi-periodicity of the integrable Boltzmann system,
  implying the applicability of KAM perturbation theory to the Boltzmann
  system with weak centrifugal force.
\end{abstract}

\maketitle
\hfill In fond memory of Boris Dubrovin

\section{Introduction}\label{s-1}
In a 1868 paper with the unpretentious\footnote{Unlike certain
  name-dropping titles.} title ``Solution of a mechanical problem''
\cite{Boltzmann_1868}, L. Boltzmann, in his search of candidate
dynamical systems obeying his Ergodic Hypothesis, introduced and
studied a simple mechanical system.  It describes of a particle moving in the
region of a plane on one side of a straight line (the wall) and
subject to a central force whose centre is not on the wall.  When the
particle hits the wall it is reflected elastically.  The force
considered by Boltzmann is the sum of an attractive one with
inverse-square law and a centrifugal force with inverse-cube law. We
refer to \cite[Appendix D]{Gallavotti_2016} for an account of
Boltzmann's paper and of its significance for the evolution of
statistical mechanics.

As first conjectured by G. Gallavotti \cite[Appendix
D]{Gallavotti_2016}, and proved by him an I. Jauslin
\cite{Gallavotti_Jauslin_2020}, the system with pure inverse-square
law has, additionally to the energy, a second independent integral of
motion and is thus far from being ergodic. One way to express the
existence of this independent integral is that the particle moves on
arcs of Kepler trajectories with one focus at the centre and the
second focus on a fixed circle, see Fig.~\ref{f-1} and Theorem
\ref{t-1}. It was thus prudent of Boltzmann to add the centrifugal
term.

It is convenient to describe the Boltzmann system as a discrete-time
dynamical system by recording the point in phase
space at each collision. The map sending a point to the point at the
next collision is called the Poincar\'e map, and the orbits are obtained
by iterating the Poincar\'e map.

\begin{figure}
  \includegraphics[width=.4\textwidth]{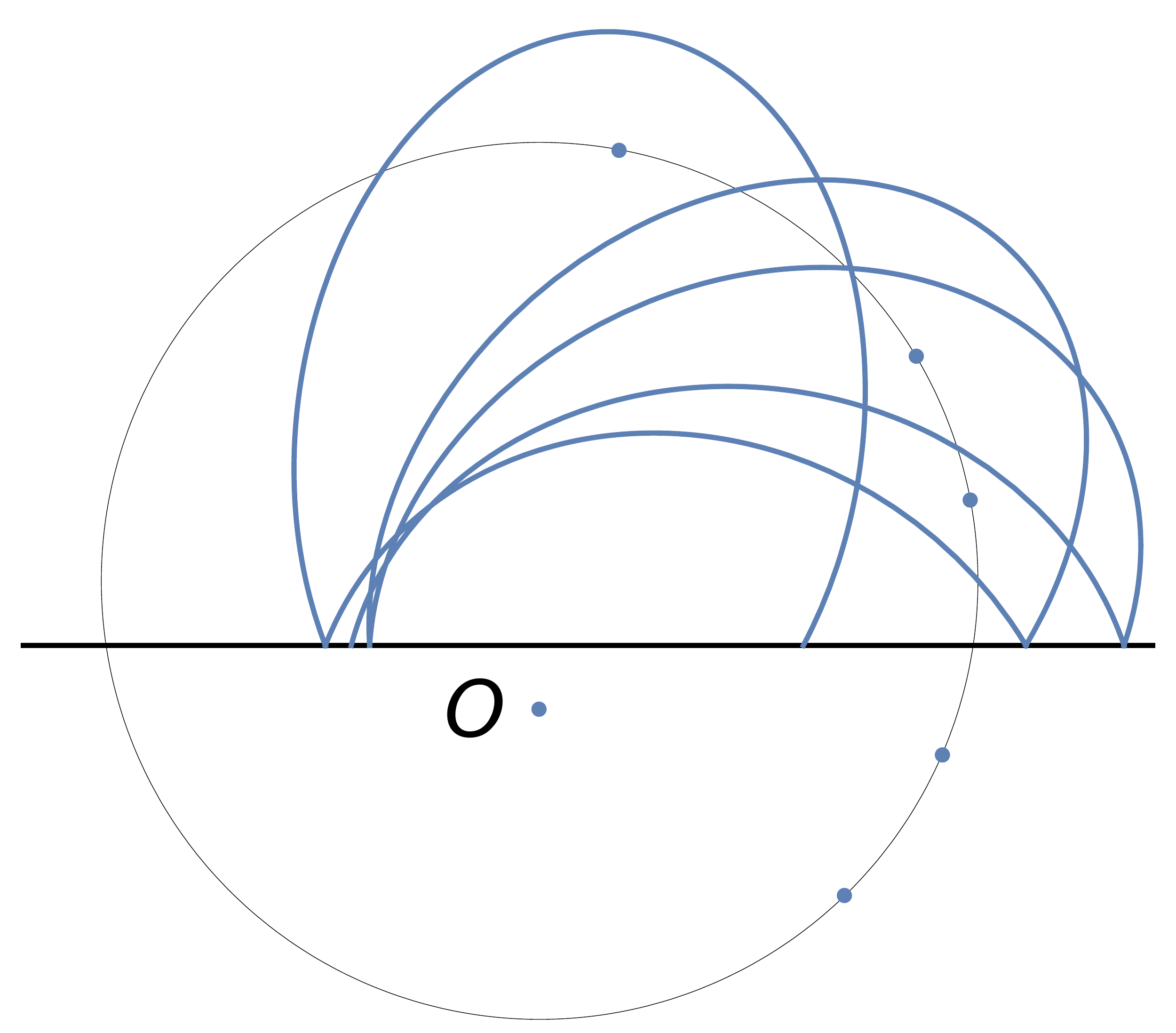}
  \caption{Trajectory of a particle subject to an attractive
    inverse-square law force bouncing off a wall (the horizontal
    line). The particle moves along elliptic arcs with one focus in
    the centre $O$ and whose second focus lies on a circle.}
  \label{f-1}
\end{figure}
In this paper we focus on this integrable case of the Boltzmann
system, with zero centrifugal force, and we show that it has the {\em
  Poncelet property}: for given values of the two integrals of motion,
either there are no periodic orbits or all orbits are
periodic. See Fig.~\ref{f-2} for an example
with period 3.

We call this the Poncelet property since it is shared by the
discrete-time dynamical system underlying the Poncelet problem.
Recall that the Poncelet problem asks for which pairs of ellipses
there is a polygon inscribed in one and circumscribing the
other. Poncelet's theorem states that if there is such a polygon for a
given pair of ellipses, then there are polygons with a vertex at an
arbitrary point of the circumscribing ellipse. Moreover all these
polygons have the same number of sides.

A beautiful explanation of the Poncelet theorem was given by
Ph. Griffiths and J. Harris \cite{Griffiths_Harris_1977},
\cite{Griffiths_Harris_1978}. Consider the space of pairs $(P,a)$
where $a$ is a tangent line to the inner ellipse and $P\in a$ is a
point of intersection of the tangent with the outer ellipse.  On this
space we have two natural involutions: $i$ maps $(P,a)$ to $(P',a)$
where $P'$ is the other point of intersection and $j$ maps $(P,a)$ to
$(P,a')$ where $a'$ is the other tangent through $P$. Taking iterates
$x$, $t(x)$, $t^2(x)$, \dots of the composition $t=j\circ i$ we obtain
a broken line consisting of chords of the outer ellipse that are
tangent to the inner one.  And a polygon is formed if and only if
$t^p(x)=x$ for some positive $p$. The observation of Griffiths and
Harris is that the (complexified) space of such pairs is naturally a
curve of genus one, and thus carries a free transitive action of its
Jacobian, an elliptic curve. It is a general fact that if we have two
non-trivial involutive automorphisms of a curve of genus one, both
having fixed points, then their composition is the translation by an
element of the elliptic curve. There are periodic orbits if and only
if this element has finite order, in which case all orbits are
periodic with the same period.

Our observation is that the integrable Boltzmann system behaves very
much in the same way: we consider for fixed generic values of the two
integrals of motion the pairs $(P,K)$ consisting of a Kepler conic $K$
and a intersection point $P\in K$ with the wall.  We show that
this space (after complexification and throwing in a couple of points
at infinity) is a smooth curve of genus one carrying two involutions.
The first
keeps the Kepler conic and changes $P$ to the other point of
intersection  with the wall and the second changes the conic to
the other conic through $P$ with the same integrals, obtained
by elastic reflection of the particle at the wall. The orbits
of the Boltzmann system are obtained by iterating the composition
$t=j\circ i$ of these involution, implying the Poncelet property.

Another consequence of this observation is the proof of a conjecture
of Gallavotti and Jauslin \cite{Gallavotti_Jauslin_2020}
that the motion is
quasi-periodic for generic values of the integrals, namely that on
generic level sets of the integrals there
is an {\em angle variable}, a map to the circle,
whose value increases by a fixed amount $\alpha$  at each iteration of $t$.
This amount, which depends on the values of the energy and the second
integral $D$, is a non-constant function of $D$.
 As discussed in
\cite{Gallavotti_2016,Gallavotti_Jauslin_2020}, this property, together
with Boltzmann's result that the map $t$ preserves an area form on the
level sets of the energy even in the presence of the centrifugal term,
allows one to apply the Kolmogorov--Arnold--Moser perturbation theory.
In this setting the relevant theorem is Moser's twist theorem
\cite{Moser_1962} implying that for small centrifugal term most of the
invariant circles on each energy surface are deformed to invariant
circles. Thus one would need a sufficiently large centrifugal term to
hope for an ergodic system.

Here we prove that the generic level sets of the integrals of motion
is diffeomorphic either to a circle or to a pair of disjoint
circles. The map $t$ is indeed mapped to a translation by an amount
$\alpha(D,E)$ whose $D$-derivative is generically non-vanishing,
possibly composed with an involution exchanging the two connected
components in the two-component case.  This implies both Conjecture 1
and 2 in \cite{Gallavotti_Jauslin_2020}.  In loc. cit. a conjectural
formula for the angle variable is given. It should also be possible to
check that conjecture with our explicit formulae.

The quasi-periodicity and the distinction between the cases of one and
two components arise from the classical theory of real elliptic curves
of Abel, Jacobi,\dots, see \cite{Du_Val_1973}.  The level set
$X_{\mathbb R}(D,E)$ is the set of real points of a smooth curve of
genus one defined over $\mathbb R$ on which (if non-empty) we have a
free and transitive action of the real points of an elliptic curve
which is a Lie group isomorphic to the circle group
$S^1=\mathbb R/\mathbb Z$ (unipartite case) or
$S^1\times\mathbb Z/2\mathbb Z$ (bipartite case).

In the unipartite case the system is (periodic or) quasi-periodic,
namely there is a diffeomorphism
$\varphi\colon S^1\to X_{\mathbb R}(D,E)$ so that
$\varphi^{-1}\circ t\circ \varphi$ is a translation
$\theta\mapsto \theta+\alpha(D,E)$ of the angle.  The diffeomorphism
can be given explicitly in terms of elliptic functions, see Theorem
\ref{t-4}. In the bipartite case the same holds for each component if
we replace the map by its iterate $t^2$. The map $t$ itself is given
by the translation by an element of the Lie group
$S^1\times \mathbb Z/2\mathbb Z$. If this element is not in the
connected component of the identity, the orbits jump back and forth
between the connected components. See Fig.~\ref{f-5} for a picture.

The transition between the unipartite and bipartite case happens when
we cross a curve in the space of parameters ($D=2$ and $D=-2$ in our
convention) where the elliptic curve degenerates to a nodal rational
curve.  We show that the unipartite case arises if $|D|<2$. If $D>2$
the real locus of the elliptic curve is bipartite and the action of
$t$ is by an element which is not in the identity component. If
$D<-2$, which happens only for sufficiently large positive energy, the
action is by an element in the identity component. Geometrically we
can understand the distinction from the constraint that the distance
between foci is smaller than the major axis, which is determined by
the energy. If $|D|>2$ all points on the circle on which the second
focus moves satisfy this constraint but if $|D|<2$ the motion is
confined to an arc of this circle.

In the next section we introduce the integrable case of the
Boltzmann system, state the main result and deduce the Poncelet
property. We also give an explicit example of a rational level set for
which all orbits have period 3. The family of
elliptic curve is constructed in Section \ref{s-3} and we discuss
the real locus and the quasi-periodicity in Section \ref{s-4}.

Recalling Boris Dubrovin's recommendation that we should not forget
that mathematics does not only consist of abstract theorems but also
of calculations, we give explicit formulas for the diffeomorphism
$\varphi$ and the rotation number $\alpha$.

\section{The integrable Boltzmann system}\label{s-2}

Let a particle subject to an attractive central force with
inverse-square law move in a plane on one side of a straight line (the
wall) not passing through the centre $O$.  When the particle hits the
wall it is reflected elastically.

To describe the system we record the position and momentum of the
particle each time it leaves the wall after a collision. We obtain a
three dimensional discrete-time dynamical system given by iterations
of the Poincar\'e map $t$ sending the position and the momentum of the
particle as it leaves the wall to the position and momentum of the
particle after the next collision.  We disregard the time it takes to go
from one collision to the next.

For clarity of exposition, let us assume for the further discussion
that the energy is negative and that the particle moves on the side of
the wall not containing the centre as in Fig.~\ref{f-1}. Then the
particle travels on arcs of ellipses with a focus in $O$, all with the
same major axis. Our discussion holds also for the case of zero or
positive energy, but one should place oneself in the projective plane
and allow the particle to wander on hyperbolae to infinity and hit the
wall from both sides.

Geometrically we can think of the Poincar\'e map as a map on pairs
$(P,K)$, where $K$ is an ellipse in the plane of motion with a focus
in $O$ and $P\in \mathcal K$ is an intersection point with the wall.
The particle at $P$ leaves the wall and follows the Kepler trajectory
$K$ until it reaches the other point of intersection $P'$.  At this point
the momentum is reflected and the particle continues on a new ellipse
$K'$ and $t(P,K)=(P',K')$. Thus the Poincar\'e map is the composition
$t=j\circ i$ of two maps
\[
  i\colon (P,K)\mapsto (P',K),
\]
exchanging the points of intersections of the ellipse with the wall
and
\[
  j\colon (P',K)\mapsto (P',K'),
\]
mapping a Kepler trajectory $K$ to the trajectory $K'$ whose
momentum at $P'$ is reflected at the wall.

The obvious fact that $i$ and $j$ are involutions will be important later.

We thus obtain a discrete dynamical system $(X_{\mathbb R},t)$, which
we call the integrable Boltzmann system, on the three-dimensional
configuration space $X_{\mathbb R}$ of pairs $(P,K)$. The maps $i,j$
make sense as birational maps in the complex domain of pairs $(P,K)$
where $K$ is a smooth conic in the two-dimensional affine space and
$P$ is an intersection point with an affine line. We thus get a
complexification $(X,t)$ where $t\colon X\dashrightarrow X$ is a
birational map.

The next observation is that the Poincar\'e map preserves a second
integral of motion in addition to the energy.

\begin{theorem}(Gallavotti--Jauslin
  \cite{Gallavotti_Jauslin_2020})\label{t-1} The Poincar\'e map of the
  integrable Boltzmann system for a particle of mass $m$ and a wall at
  distance $h$ to the centre preserves the energy $E$ and the
  combination $D=L^2-2 h A_\perp$ of the angular momentum $L$ and the
  component $A_\perp$ of the
  Laplace--Runge--Lenz vector perpendicular to the wall.
\end{theorem}
Recall that the  (Hermann--Bernoulli--)
Laplace--Runge--Lenz vector is the conserved quantity
$\mathbf A=\mathbf p\times \mathbf L-m\kappa\,
\mathbf{x}/|\mathbf{x}|$ of the Kepler problem with Hamiltonian
$|\mathbf p|^2/2m-\kappa/|\mathbf x|$. In terms of trajectories
$\mathbf A$ is a vector in the plane of motion along the major axis
whose length is the eccentricity times $m\kappa$.

Note that $i$ preserves all conserved quantities
$E,\mathbf L,\mathbf A$ of the Kepler problem. Therefore in fact both
involution $i$ and $j$ preserve $E$ and $D$.

We also observe that because of the classical relation
$2mEL^2=|\mathbf A|^2-m^2\kappa^2$ between conserved quantities of the
Kepler problem, $D$ can be written as
\[
  D=\frac{|\mathbf A|^2}{2mE}-2h\, A_\perp-\frac {m\kappa^2}{2E},
\]
so that the Laplace--Runge--Lenz vector moves on a circle of radius
\[
  R=\sqrt{m^2\kappa^2+2mDE+4h^2 m^2E^2}.
\]
More geometrically, noticing that
$\mathbf A/mE$ is the vector connecting the centre to the other
focus, we can say that the particle moves along arcs of Kepler
ellipses whose second focus lies on a circle of radius $R/m|E|$
centred at the mirror image of the centre with respect to the
wall.~\footnote{In \cite{Gallavotti_Jauslin_2020} the authors consider
  the case of a particle of mass $m=1$ in a potential with coupling
  constant $\kappa=2\alpha$ and energy $E=A/2$; $D$ is
  called $R$ there. The radius of the circle on which the midpoint of
  the ellipse moves ($R/2|E|m$ in our notation) is denoted by
  $R_0$ in loc.~cit.}

From now on we choose for convenience units of time, length and mass
so that $\kappa=1$, $h=1$, $m=1$.

The theorem is reduced in
\cite{Gallavotti_Jauslin_2020} to a geometric theorem on
ellipses. Here is an alternative, possibly more direct proof.  We can
assume that the motion takes place in the plane with coordinates
$x_1,x_2$ with the centre at the origin and the wall at $x_2=1$. The
phase space has coordinates $x_1,x_2,p_1,p_2$ and the integrals of motion
of the Kepler problem are
\[
  E=\frac{p_1^2+p_2^2}{2}-\frac 1{r}, \quad L=x_1p_2-x_2p_1, \quad
  A_1=p_2L-\frac{x_1}{r}, \quad A_2=-p_1L-\frac{x_2}{r}.
\]
Here $r=\sqrt{x_1^1+x_2^2}$ is the distance to the centre.  In these
coordinates,
\[
  D=L^2-2A_2=\frac{A_1^2+A_2^2-1}{2E}-2A_2.
\]
When the particle hits the wall at a point with $x_2=1$, the sign of
$p_2$ is changed. Thus the angular momentum $L$ and the orthogonal
component $A_2$ of the Laplace--Runge--Lenz vector change to $L'=-L-2p_1$,
$A_2'=-p_1L'-1/r=A_2+2p_1L+2p_1^2$, respectively. Therefore
$(L')^2-2A_2'=L^2-2A_2$, as claimed.

The complexified configuration space $X$ is thus foliated by the level
sets $X(D,E)$ of the integrals of motion.
\begin{theorem}\label{t-2}
  Let $D,E\in\mathbb C$ such that $D^2\neq 4$, $1+2ED+4E^2\neq 0$,
  $D+2E\neq0$.  Then the level set $X(D,E)$ has a compactification
  $\bar X(D,E)$ which is a smooth projective curve of genus 1, and $i$
  and $j$ extend to automorphisms with fixed points.
\end{theorem}
The construction of the compactification and the proof of this theorem
is presented in the next section, see Theorem \ref{t-3}.

Thus we have a curve of genus one with two holomorphic involutions and
we can reproduce the argument of Griffiths and Harris on the Poncelet
theorem: any smooth curve of genus one has a free transitive action of
a group, the associated elliptic curve. The composition of two
involutions with fixed points is the action an element of this
group. More explicitly, by uniformization we have an isomorphism
\[
  \bar X(D,E)\to \mathbb C/\Lambda,
\]
for some lattice $\Lambda=\mathbb Z\omega_1+\mathbb Z\omega_2$.  Any
holomorphic automorphism of $\mathbb C/\Lambda$ is of the form
$z\mapsto az+b\operatorname{mod}\Lambda$. An involution has
$a=\pm1$. A non-trivial involution with fixed points has $a=-1$. Thus
the composition of two non-trivial involutions $z\mapsto -z+b$,
$z\mapsto -z+c$ is the translation $z\mapsto z+v$ by
$v=c-b\in\mathbb C/\Lambda$.
This implies the following result.
\begin{corollary}\label{c-1}
  Let $D,E\in\mathbb C$ obey the assumptions of Theorem \ref{t-2}. Then
  $t$ is the action of an element
  of the associated elliptic curve. In particular there is a
  biholomorphic map $\varphi\colon \mathbb C/\Lambda \to\bar X(D,E)$
  for some lattice $\Lambda\subset \mathbb C$ such that
  $\varphi^{-1}\circ t\circ \varphi(u)=u+T$ for all
  $u\in\mathbb C/\Lambda$ and some $T=T(D,E)\in\mathbb C/\Lambda$.
\end{corollary}
One consequence of this is the
Poncelet property. The sequence $(t^n(x))_{n\in\mathbb Z}$ of images
of $x\in\bar X(D,E)$ of iterates of $t$ is called the orbit through
$x\in \bar X(D,E)$. An orbit is called periodic if $t^p(x)=x$ (and
thus $t^{n+p}(x)=t^n(x)$ for all $n$) for some positive integer $p$.
The minimal such $p$ is called the period of the orbit.
\begin{corollary}\label{c-2} 
  If for some $D,E$ obeying the assumptions of Theorem \ref{t-2},
  $\bar X(D,E)$ has a periodic orbit then all
  orbits in $\bar X(D,E)$ are periodic and they all have the same
  period.
\end{corollary}
The orbits have period $p$ if $T(D,E)$ has order $p$, i.e., if
$p$ is minimal such that $p\,T(D,E)\equiv 0\operatorname{mod}\Lambda$.

Thus we can determine the pairs $(D,E)$ for which all orbits have
period $p$ by solving the equation $t^p(x)=x$ where $x\in X(D,E)$ is
arbitrary. We get equations of countably many algebraic curves in the
$(D,E)$ plane for which all orbits are periodic.  The case $p=1$ where
$t$ is the identity occurs in the degenerate range of parameter
$D+2E=0$. In this case all Kepler trajectories are tangent to the
wall. An orbit of period $p=2$ corresponds to a conic intersecting the
wall at right angle at both points of intersection. This arises only
if $1+2ED+4E^2=0$, which is excluded by the theorem. In this degenerate case
the real locus
$X_{\mathbb R}(D,E)$ consists of two points consisting of a conic
which is symmetric with respect to the wall with its two intersection
points.  The first interesting case is $p=3$: if $(D,E)$ lies the
algebraic curve
\[
  4(D^2-4)E^2+4D(D^2-3)E+D^4-2D^2-3=0,
\]
for example if $D=7/4,E=-5/24$, all orbits in $\bar X(D,E)$ are
periodic with period $3$, see Fig.~\ref{f-2}.
\begin{figure} 
  \includegraphics[width=.4\textwidth]{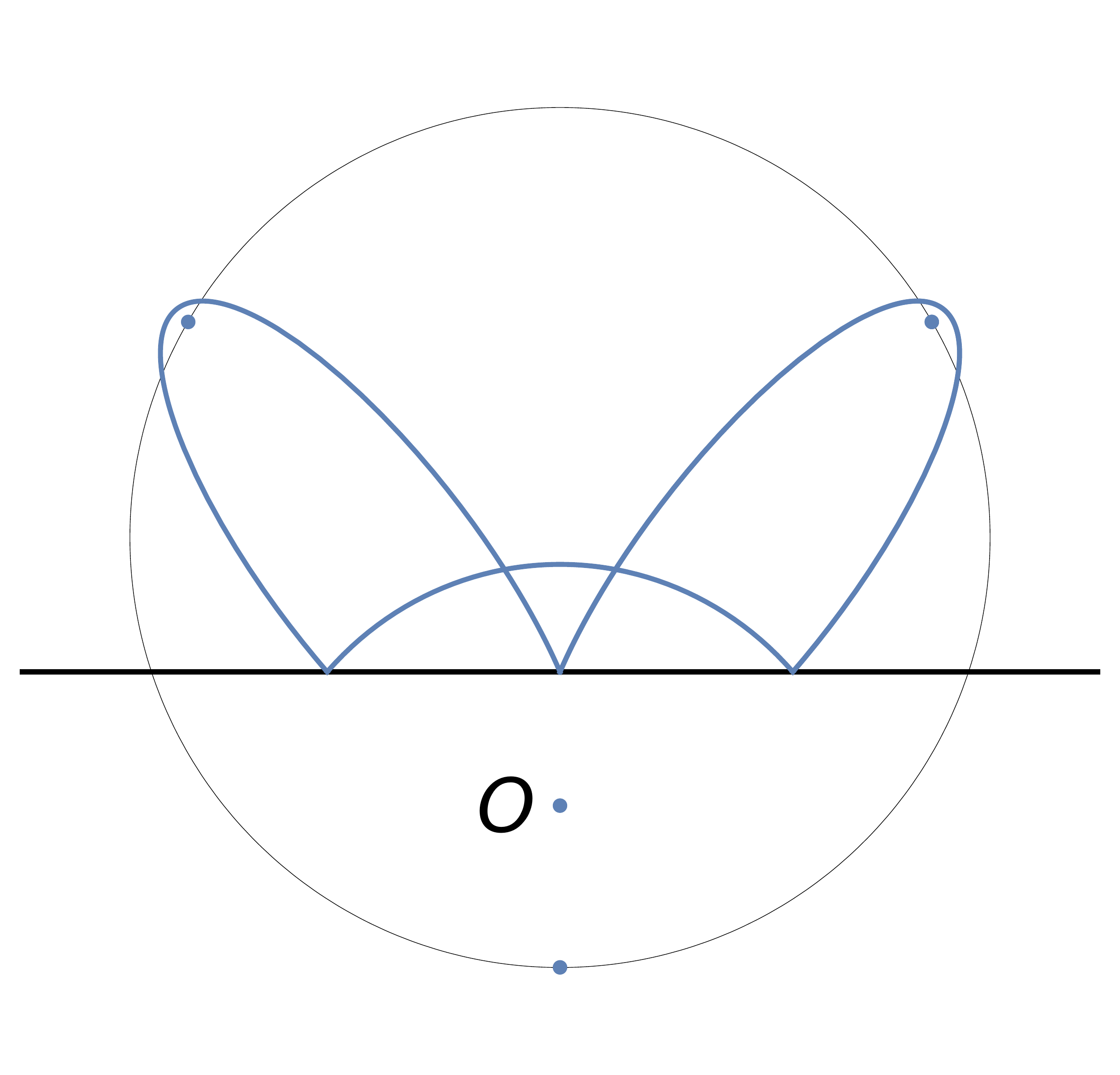}
  \includegraphics[width=.4\textwidth]{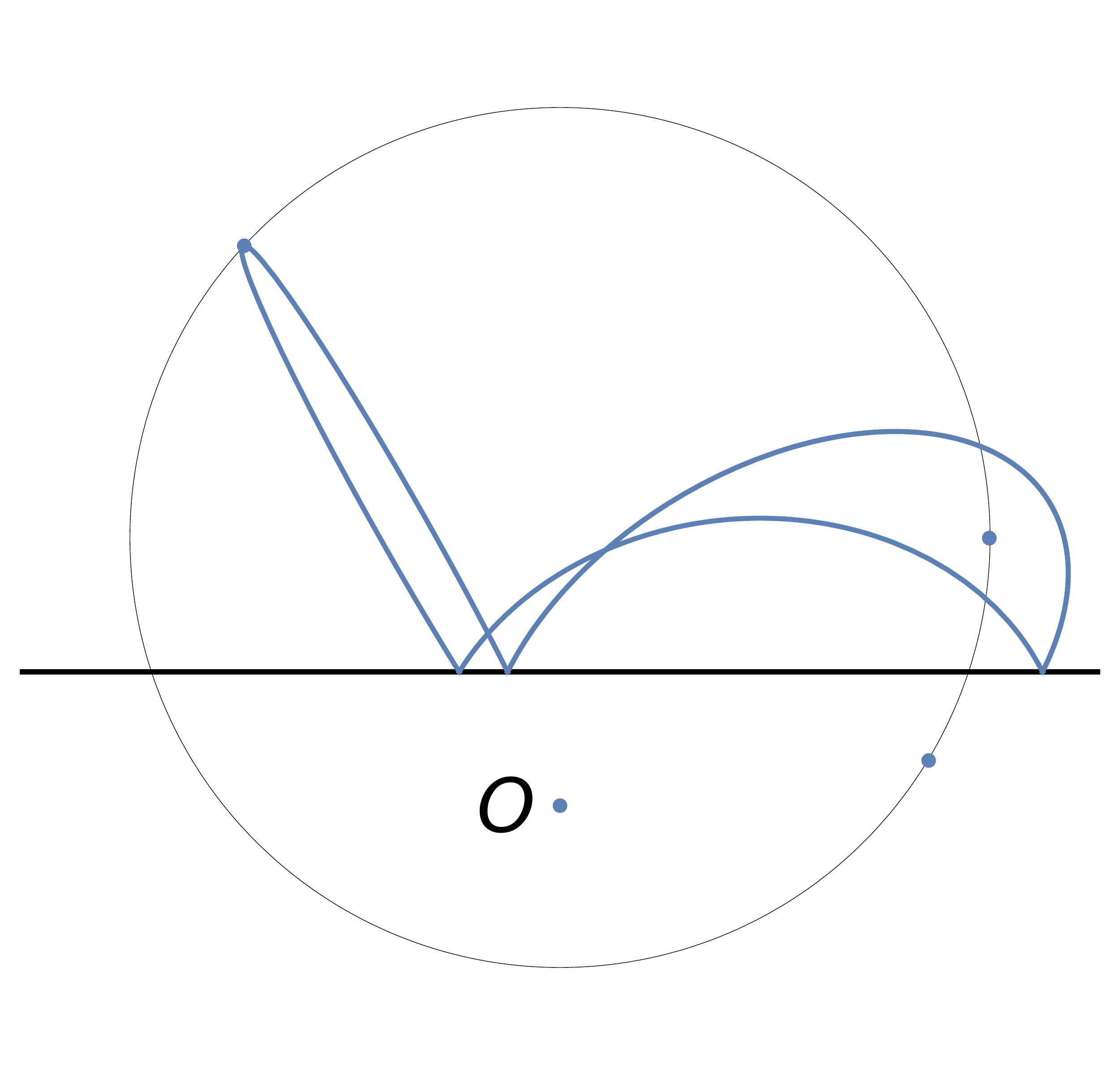}
  \caption{Two periodic orbits with integrals $D=7/4$, $E=-5/24$ in
    natural units.}
  \label{f-2}
\end{figure}

We conclude this section by discussing qualitatively the degenerate
cases not covered by the theorem, see Theorem \ref{t-3} for more
detailed statements.  As mentioned above, if $D+2E=0$ then $\bar X(D,E)$
consists of Kepler conics tangent to the wall. It is a rational curve.
If $D=\pm 2$ the elliptic curve degenerates to a
nodal curve with one node.  The node is a fixed point corresponding to
a degenerate conic (double line) orthogonal to the wall. The generic
orbit $t^j(x)$ converges to it both for $j\to\infty$ and
$j\to -\infty$.  Finally for $1+2ED+4E^2=0$ we get a nodal curve with
two irreducible components and $t$ maps one node to the other. The
nodes represent a conic which is symmetric with respect to the wall
with its two intersection points. In the latter two cases the Poncelet
property does not hold.
\section{A family of elliptic curves}\label{s-3}
We parametrize a point $(P,K)\in X(D,E)$ by the first coordinate $x$
of $P=(x,1)$ and the Laplace--Runge--Lenz vector $(A_1,A_2)$. The
corresponding Kepler trajectory (for $E<0$) is then determined by the
properties that the the foci are $(0,0)$, $(A_1/E,A_2/E)$ and the
major axis is $2a=-1/E$, see, e.g., \cite[Section
II]{van_Haandel_Heckmann_2009}.  From the condition that the sum of
distances to the foci is $2a$ we deduce the equation of the Kepler
ellipse
\[
  x_1^2+x_2^2=\left(L^2-A_1x_1-A_2x_2\right)^2, \quad
  L^2=\frac{A_1^2+A_2^2-1}{2E}.
\]
Since $D=L^2-2A_2$ this can be written as
\[
  x_1^2+x_2^2=\left(2A_2+D-A_1x_1-A_2x_2\right)^2.
\]
This also holds for $E\geq0$ with a similar derivation. Setting $x_2=1$
gives the equation for the two points of intersection of the Kepler
trajectory with the wall.

Thus $X(D,E)$ is defined as the algebraic set in the affine space with
coordinates $(x,A_1,A_2)$ by the equations
\begin{align}
  A_1^2+A_2^2-4EA_2&=1+2DE,\label{e-1}
  \\
  x^2+1&=(A_2+D-A_1x)^2.\label{e-2}
\end{align}
It will be useful, especially for the study of the real locus, to
introduce the parameter $R^2=1+2DE+4E^2$, so that \eqref{e-1}
describes a circle of radius $R$, and write $\eqref{e-1}$ as
\begin{equation}\label{e-1'}
  A_1^2+(A_2-2E)^2=R^2.
\end{equation}
The formulae for the involutions are
\begin{align}
  i(x,A_1,A_2)&=\left(-\frac{2(A_2+D)A_1}{1-A_1^2}-x,A_1,A_2\right),\label{e-i}
  \\
  j(x,A_1,A_2)&=(x,A_1',A_2'),\notag
  \\
  A_1'&=\frac{x^2-1}{x^2+1}A_1-\frac{2x}{x^2+1}A_2+\frac{4xE}{x^2+1},\label{e-j}
  \\
  A_2'&=-\frac{2x}{x^2+1}A_1-\frac{x^2-1}{x^2+1}A_2+\frac{4x^2E}{x^2+1}.\notag
\end{align}
The involution $i$ exchanges the two solutions of \eqref{e-2}. The
second can be deduced from the calculation of how the Laplace--Runge--Lenz
vector changes at a reflection, as in the proof of Theorem
\ref{t-2}. We need to know two things from these formulae for the
proof of Theorem \ref{t-2}:
\begin{enumerate}
\item[(a)] $i$ and $j$ are rational map, i.e., maps that are defined on
  some dense open set and are given by rational functions of the
  coordinates. Since $i,j$ are involutions (where defined) they are
  moreover birational maps, namely rational maps with rational
  inverses.
\item[(b)] Both have fixed points (in the complex domain). For $i$ they
  correspond to degenerate Kepler trajectories (double lines) and
  occur when the quadratic equation \eqref{e-2} for $x$ has double
  roots. The fixed points of $j$ are points $(x,A_1,A_2)$ so that the
  Kepler conic parametrized by $(A_1,A_2)$ meets the wall at $x$ at a
  right angle.  The values of the coordinates $(x,A_1,A_2)$
  at the fixed point can be easily computed. For $i$
  they are the solutions of \eqref{e-1}
  such that $A_2=-D/2$. For $j$ they are the solutions such that
  $A_2=2E/(1\pm R)$.
\end{enumerate}
We turn to the description of the compactification of the space
$X(D,E)$ of solutions of \eqref{e-1}, \eqref{e-2}. 

The first equation \eqref{e-1} describes a conic, which is smooth
provided $R^2=1+2ED+4E^2\neq 0$. For each point in this conic there are
generically two solutions $x$ of the second equation \eqref{e-2}. It
is useful to change variables by completing the square and setting
$z=(1-A_1^2)x+A_1(A_2+D)$ so that the second equation becomes
\begin{equation}\label{e-2'}
  z^2=A_1^2+(A_2+D)^2-1.
\end{equation}
Physically $z=\sqrt{D+2E}\,L$ is proportional
to the angular momentum $L$.
  
In projective coordinates $A_1=\alpha_1/\alpha_0,A_2+D=\alpha_2/\alpha_0,
z=\zeta/\alpha_0$
the equation are
\begin{align}
  \alpha_1^2+\alpha_2^2-2(D+2E)\alpha_0\alpha_2
  &=(1-(D+2E)D)\alpha_0^2,\label{e-3}
  \\
  \zeta^2&=\alpha_1^2+\alpha_2^2-\alpha_0^2.\label{e-4}
\end{align}
These equations define a family of curves in $\mathbb{CP}^3$ with
projective coordinates $(\alpha_0:\alpha_1:\alpha_2:\zeta)$
over the affine plane with coordinates $D,E$. The fibre $\bar X(D,E)$
contains $X(D,E)$ as an open dense subset. The map
$p\colon (x,A_1,A_2)\mapsto (A_1,A_2)$ to the set of solution of \eqref{e-1}
extends to a two-sheeted covering $p\colon \bar X(D,E)\to C(D,E)$ to the
plane curve defined by \eqref{e-3}. It is the restriction to
$\bar X(D,E)$ of the central projection $\mathbb {CP}^3\smallsetminus
\{c\}\to\mathbb {CP}^2$ (a.k.a.~the line bundle $O(1)$)
with centre $c=(0:0:0:1)$ onto the plane $\zeta=0$.

The ramification locus is the intersection of 
\eqref{e-3} and the quadric $\alpha_1^2+\alpha_2^2-\alpha_0^2=0$ and consists
generically of four points, so that $\bar X(D,E)$ is a generically
a smooth projective curve of genus one. 

If we exclude the line $D+2E=0$, where the two quadrics \eqref{e-3}, \eqref{e-4} coincide,
corresponding to the case where all Kepler trajectories are tangent to
the wall, the number of ramification points is at most 4. They can be
readily computed to be the points ``at infinity'' $P_\pm=(0:1:\pm i:0)$ and
\[
  \left (1:\pm\sqrt{1-\frac{D^2}4}:\frac D2:0\right).
\]
The four ramification points are distinct if and only if $D\neq \pm2$.

It will be useful to compute the action of
$t$ on the  fixed points of $i$ at infinity, which belong to all
curves $\bar X(D,E)$.
\begin{lemma}\label{l-1} Assume that $D+2E\neq 0$, and let
  $P_{\pm}=(0:1:\pm \I:0)\in \bar X(D,E)$.  Then the image
  $Q_\pm=t(P_\pm)$ of $P_\pm$ under $t=j\circ i$ has projective
  coordinates
  \[
    Q_\pm =\left(1:\pm\I\,\left(-\frac D2+\frac1{2(D+2E)}\right):
      \frac D2-\frac1{2(D+2E)}:\pm\I\right).
  \]
\end{lemma}
\begin{proof} Since $i(P_\pm)=P_\pm$, $t(P_\pm)=j(P_{\pm})$. A
  simple way to check the claim is to check that $Q_\pm$ belong to
  $\bar X(D,E)$, which is straightforward, and prove that
  $j(Q_\pm)=P_\pm$. By writing $Q_{\pm}=(1:A_1:A_2+D:z)$ we see that
  $A_1=\mp i(A_2+D)$ so the relation between $z$ and $x$ (see the
  discussion leading to \eqref{e-2'}) becomes
  $z=(1-A_1^2)x\pm\I\,A_1^2$.  Thus $z=\pm \I$ corresponds to
  $x=\pm \I$.  Then taking the limit $x\to \pm\I$ in the formula
  \eqref{e-j} for $j$, we see that the ratio $(A_2+D)/A_1$ tends to
  $\pm\I$ so that $j(Q_{\pm})=P_{\pm}$.
\end{proof}
\begin{theorem}\label{t-3}
  Let $D,E\in\mathbb C$ be such that $D+2E\neq0$.
  \begin{enumerate}
  \item[(i)] If $1+2DE+4E^2\neq 0$ and $D^2\neq 4$, then the closure
    $\bar X(D,E)$ of $X(D,E)$ in $\mathbb {CP}^3$ is a smooth
    projective curve of genus one. The birational maps $i$ and $j$
    extend to non-trivial involutive automorphisms of $\bar X(D,E)$,
    both having fixed points. Their composition $t=j\circ i$ is a
    non-trivial element of the elliptic curve.
  \item[(ii)] If $D^2=4$, then $\bar X(D,E)$ is a rational curve with
    one node, which a fixed point for both involutions. Their composition
    is a non-trivial automorphism.
  \item[(iii)] If $1+2DE+4E^2=0$ and $D^2\neq 4$ then $\bar X(D,E)$ has
    two components meeting at two nodes. The involution $i$ preserves
    the components and permutes the nodes, $j$ permutes the components and
    fixes the nodes. 
  \end{enumerate}
\end{theorem}
\begin{proof}
  The claims of non-triviality of the automorphims follow from Lemma
  \ref{l-1}: for $D+2E\neq0$, $j$ and $t$ map the point at infinity
  $P_+$ to a finite point $Q_+$.

  (i) The first claim follows from the fact that any two-sheeted cover
  of a smooth rational curve with four simple ramification points is a
  curve of genus 1. The second claim is a consequence of the fact that
  any birational map between smooth projective curves is an
  isomorphism, see \cite[Chapter I, Proposition 6.8]{Hartshorne_1977}.

  (ii) If $D^2=2$ two of the four ramification points merge at the
  fixed point $(x,A_1,A_2)=(0,0,-1)$ of both $i$ and $j$.

  (iii) Here the base curve given by \eqref{e-1} consists of two lines
  $\ell_\pm$ with equations $A_2-2E=\pm iA_1$ meeting at the node
  $(A_1,A_2)=(0,2E)$. There are two ramification points on each line
  and they are distinct and distinct from the singular point as long
  as $D^2-4$ (the $A_1$-coordinates of the finite ramification points
  is $\pm\sqrt{1-D^2/4}$).  Thus $\bar X(D,E)$ is the union of two
  smooth rational curves $C_{\pm}$, two-sheeted coverings of
  $\ell_{\pm}$. They meet at two nodes, the preimages of the singular
  point $(0,2E)$. The sheets and thus the nodes are interchanged by $i$.
  The nodes represent Kepler conics with foci
  $(0,0)$ and $(0,2)$, and their intersection points with the
  wall. The conics are symmetric with respect to the reflection at the wall.
  Therefore they cut the wall at straight angles and they are
  fixed by $j$. Since $D+2E=-1/2E$, the image of $P_+$, which
  belongs to the component $C_+$, is $(1:A_1:A_2+D:\I)$ with $A_1=\I(-D/2-E)$,
  $A_2=-D/2+E$. Thus $j(P_+)=t(P_+)\in C_-$ and $j$ is an isomorphism
  $C_+\to C_-$.  
\end{proof}  
This implies Theorem \ref{t-2}. In particular in the smooth case
(i) there is
on $\bar X(D,E)$ an abelian differential $\lambda$, unique
up to normalization, $\Lambda$ is the lattice of integrals of $\lambda$
over closed curves and, for any choice of base point
$P_0\in \bar X(D,E)$ the Abel--Jacobi map
$P\mapsto\int_{P_0}^P\lambda$ is a biholomorphic map
$\bar X(D,E)\to \mathbb C/\Lambda$.

Under this map, $t$ becomes the translation by an element
$T\in\mathbb C/\Lambda$. It can be computed by applying $t$ to any
point $P\in \bar X(D,E)$ for instance the point at infinity
$P_+=(0:1:\I:0)$ and taking the difference of the images of
$Q_+=t(P_+)$ and $P_+$:
\[
  T=\int_{P_+}^{Q_+}\lambda.
\]
More abstractly, $\bar X(D,E)$ is the fibre of an algebraic family
$\bar X\to Y=\mathbb A^2\smallsetminus H$ of projective curves of
arithmetic genus one over the complement of the line $H=\{D+2E=0\}$
with two sections $P_+,Q_+=t\circ P_+\colon Y\to
X$. The action of $t$ on the fibres is the action of the element of
the Jacobian defined by the divisor $P_+-Q_+$.

Here is an explicit parametrization of $\bar X(D,E)$ by Jacobi
elliptic functions. Let $R$ be a square root of $R^2=1+2DE+4E^2$.
A rational parametrization of the space of solutions of \eqref{e-1'}
is
\begin{equation}\label{e-5}
  A_1=\frac{2\I R  s}{1-s^2},\quad A_2=2E+R\frac{1+s^2}{1-s^2}.
\end{equation}
and the real points for $D,E$ real correspond to imaginary $s$.

Then the second equation, in the form \eqref{e-2'},
is brought to the standard Legendre form
\[
  y^2=(1-s^2)(1-k^2 s^2),\quad     k^2=\frac{D+4E-2R}{D+4E+2R},
\]
by introducing the variable
\[
  y=C^{-1}z(1-s^2),\quad \text{where $C^2=(D+2E)(D+4E+2R)$}.
\]
In these variables the involution $i$ is $(s,y)\mapsto (s,-y)$ and
a non-zero differential is $\lambda=ds/y$.

The coordinates of $P_+,Q_+$ are 
\[
  s(P_+)=-1,\quad y(P_+)=0,\quad   s(Q_+)  =\frac{D+2E+R}{D+2E-R},
\quad  y(Q_+)=\frac{-4\I (D+2E)R}{C(D+2E-R)^2}.
\]

The lattice $\Lambda$ is spanned by $4K$ and $2 {\I} K'$ where
$K,{\I} K'$ are the complete elliptic integrals of the first kind
\begin{equation}\label{e-K}
  K=\int_{0}^{1}\frac{ds}{\sqrt{(1-s^2)(1-k^2s^2)}},
  \quad {\I} K'=\int_1^{1/k}\frac{ds}{\sqrt{(1-s^2)(1-k^2s^2)}}.
\end{equation}
We then have the parametrization in terms of Jacobi elliptic
functions, see, e.g., \cite[Chapter 4]{Du_Val_1973},
\[
  s=\sn(u,k), \quad y=\frac d{du}\sn(u,k)
  =\cn(u,k)\dn(u,k),
\]
providing an isomorphism $\mathbb C/\Lambda\to \bar X(D,E)$.

Thus we get an explicit uniformization. Instead of the $x$-coordinate
$x$ of the hitting point it is convenient to use the variable $z$
(proportional to the angular momentum),
related to it via \eqref{e-2'}:
\[
 z=    (1-A_1^2)x+A_1(A_2+D).
\]
\begin{theorem}\label{t-4}
  Let $(D,E)$ obey the hypotheses of Theorem \ref{t-2} and choose
  a square root of $R^2=1+2ED+4E^2$. Set
  \[
    k^2=\frac{D+4E-2R}{D+4E+2R},
    \quad s_0=\frac{D+2E+R}{D+2E-R}.
  \]
  Let $\Lambda=\mathbb Z 4K+\mathbb Z 2 {\I} K'$ where
  $K=K(k^2),K'=K'(k^2)$ are the complete elliptic integrals
  \eqref{e-K} with Legendre modulus $k$.  There is a biholomorphic map
  $\varphi\colon\mathbb C/\Lambda\to \bar X(D,E)$ given by
  \[
    A_1=2{\I}R \frac{\sn u}{\cn^2u}, \quad A_2=2E-R+\frac{2R}{\cn^2 u},
    \quad
    z=C \frac{\dn u}{\cn u}.
  \]
  Here $\sn,\cn,\dn$ are the classical Jacobi elliptic function with
  Legendre modulus $k$ and $C^2=(D+2E)(D+4E+2R)$.  Moreover
  $\varphi^{-1}\circ t\circ \varphi$ is the translation $u\mapsto u+T$
  by
  \[
    T=\int_{-1}^{s_0}\frac{ds}{\sqrt{(1-s^2)(1-k^2 s^2)}}\mod\Lambda,
  \]
\end{theorem}
\begin{remark}\label{r-1}
  There is an ambiguity of sign in the definition of the shift $T$ and
  the choice of square root of $C$. The condition fixing the sign of
  $T$ is that the value of the square root in the integrand at $s=s_0$
  is $y(Q_+)$. We will be more precise in the real case below.
\end{remark}

\section{The real locus}\label{s-4}
So far we considered the problem in the complex domain. Here we want
to restrict to the physical region and discuss the quasi-periodicity.
For each real values $D,E$ of the integrals of motion takes place in
the set $X_{\mathbb R}(D,E)$ of solutions of \eqref{e-1},\eqref{e-2}.

From the previous section we know that $X_{\mathbb R}(D,E)$, with
$D,E$ real obeying the hypotheses of Theorem \ref{t-2} has a
compactification $\bar X_{\mathbb R}(D,E)$ which is the set of real
points of a smooth complex projective curve $\bar X(D,E)$ defined over
$\mathbb R$.  The complex conjugation in $\mathbb {CP}^3$
restricts to an antiholomorphic involution $\sigma$ of $\bar X(D,E)$
and $\bar X_{\mathbb R}(D,E)$ is the set of fixed points of $\sigma$.
It can be empty or have 1
or 2 connected components.  There is one connected components if and
only if exactly two ramification points are real.  Under the
uniformization isomorphism $\bar X(D,E)\to\mathbb C/\Lambda$, $\sigma$
is carried to an antiholomorphic involution of $\mathbb C$ preserving
the lattice $\Lambda$, which we can choose to be $u\mapsto -\bar u$.

In our case the ramification points at infinity $(0:1:\pm i:0)$ are
not real, so we have the following two possibilities:
\begin{enumerate}
\item[I.] {$|D|<2$.} Exactly two ramification points are real,
  $\bar X_{\mathbb R}(D,E)$, if non-empty, has one connected
  component.  This is the {\em unipartite case}: the lattice $\Lambda$
  is {\em rhombic} namely it is generated by two periods
  $\omega,-\bar\omega$ which are mapped to each other by the
  antiholomorphic involution. The real locus $\mathcal E(\mathbb R)$
  in $\mathcal E=\mathbb C/\Lambda$ is isomorphic to the circle group
  $\mathbb R/\mathbb Z$ via
  \begin{equation}\label{e-7}
    \theta\mapsto  (\omega-\bar\omega)\theta\mod \Lambda,
    \quad\theta\in\mathbb
    R/\mathbb Z.
  \end{equation}
\item[II.]  {$|D|>2$.} No ramification point lies on the real locus,
  $\bar X_{\mathbb R}(D,E)$, if non- empty, has two connected
  components (the {\em bipartite case}).  The lattice $\Lambda$ is {\em
    rectangular}, namely it is generated by a period $\omega_1$ which
  changes sign under the antiholomorphic involution and a period
  $\omega_2$ which is fixed by it. The real locus
  $\mathcal E(\mathbb R)$ in $\mathbb C/\Lambda$ is isomorphic to the
  group $\mathbb R/\mathbb Z\times\mathbb Z/2\mathbb Z$ via
  \begin{equation}\label{e-8}
    (\theta,\epsilon)\mapsto \omega_2\theta
    +\frac{\omega_1}2\epsilon\mod\Lambda,
    \quad (\theta,\epsilon)\in\mathbb R/\mathbb Z\times\mathbb Z/2\mathbb Z.
  \end{equation}
\end{enumerate}
Thus we have the following result.
\begin{theorem}\label{t-5}
  Let $D,E\in\mathbb R$ obey the hypotheses of Theorem \ref{t-2} and
  assume that $X_{\mathbb R}(D,E)$ is non-empty.  Let
  $\mathcal E\cong\mathbb C/\Lambda$ be the elliptic curve associated
  with $X(D,E)$.  Then
  $t\colon\bar X_{\mathbb R}(D,E)\to \bar X_{\mathbb R}(D,E)$ is the
  action of an element of the group $\mathcal E(\mathbb R)$ of real
  points, which is isomorphic to the circle group $S^1$ if $|D|<2$ and to
  $S^1\times \mathbb Z/2\mathbb Z$ if $|D|>2$.
\end{theorem}

We next examine the condition for the real locus to be non-empty,
give an explicit parametrization of $X_{\mathbb R}(D,E)$ and identify
the element of $\mathcal E(\mathbb R)$ by which $t$ acts.

We first notice that we have a bijection
$X_{\mathbb R}(D,E)\to X_{\mathbb R}(-D,-E)$ sending $(A_1,A_2,x)$ to
$(-A_1,-A_2,x)$. This bijection commutes with the involutions and
changes the sign of $L^2=D+2A_2$. We can assume that $D+2E>0$, which
is the physical region where the angular momentum is real. Indeed it
follows from \eqref{e-1'} and the inequality $A_1^2+(A_2+D)^2\geq 1$,
guaranteeing the existence of real solutions $x$ of \eqref{e-2} that
$(D+2E)(D+2A_2)\geq 0$ and $L^2=D+2A_2$ is non-negative on
$X_{\mathbb R}(D,E)$ only if $D+2E>0$.

A necessary condition for $X_{\mathbb R}(D,E)$ to be non-empty
is that $1+2ED+4E^2\geq0$, see \eqref{e-1'}. It is then convenient
to introduce the radius as the non-negative square root
\[
  R=\sqrt{1+2ED+4E^2}.
\]
\begin{lemma}\label{l-3} Assume $D,E\in\mathbb R$ obey $R>0$, $D+2E>0$.
  Then $X_{\mathbb R}(D,E)$ is
  non-empty if and only if $D+4E+2R>0$.
\end{lemma}  
\begin{proof}
  The configuration space is defined by the two equations \eqref{e-1},
  \eqref{e-2}.
  The second equation has real solutions for $x$ if and only if the
  discriminant $A_1^2+(A_2+D)^2-1$ is non-negative. Thus
  $X_{\mathbb R}(D,E)$ is non-empty if and only if there exists
  $(A_1,A_2)\in\mathbb R^2$ such that
  \begin{align*}
    A_1^2+(A_2-2E)^2&=R^2,
    \\
    A_1^2+(A_2+D)^2&\geq1.    
  \end{align*}
  Taking the difference we can replace the inequality by
  $(D+2E)(D+2A_2)\geq0$ or $A_2\geq -D/2$. Given $A_2$ obeying this
  inequality, we find $A_1$ if and only if $(A_2-2E)^2\leq R^2$.
  These two inequalities for $A_2$ can be simultaneously satisfied if
  and only if $D+4E+2R\geq 0$. We still have to show that this
  inequality cannot be an equality.  If $D+4E=-2R$ then taking squares
  gives
  $D^2=4$, which is excluded by the assumption of Theorem \ref{t-4}.
\end{proof}
In terms of the Legendre coordinates the real points correspond to
imaginary $s$ and real $y$. The involution is $s\mapsto -\bar s$,
$y\mapsto \bar y$. Thus $u=\int ds/y$ is mapped to $-\bar u$ under
the antiholomorphic involution. The periods $\omega,\omega_1,\omega_2$
can be expressed in terms of complete elliptic integrals as follows,
see Fig.~\ref{f-3}.
 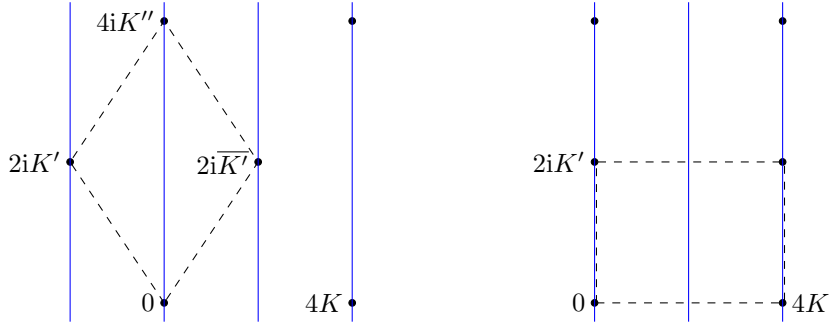
\begin{figure}
 \begin{tikzpicture}[scale=2.5]
  \filldraw (0,0) circle (.5pt) node [below,left,black] {$0$};
  \filldraw (1,0) circle (.5pt) node [below,left,black] {$4K$};
  \filldraw (.5,.75) circle (.5pt) node [below,left,black] {$2\I \overline{K'}$};
  \filldraw (0,1.5) circle (.5pt) node [below,left,black] {$4\I K''$};
  \filldraw (-.5,.75) circle (.5pt) node [below,left,black] {$2\I {K'}$};
  \filldraw (1,1.5) circle (.5pt);
  \draw [dashed] (0,0) -- (-.5,.75) -- (0,1.5) -- (.5,.75) -- (0,0);
  \draw [blue] (0,-0.1) -- (0,1.6);
  \draw [blue] (-0.5,-0.1) -- (-0.5,1.6);
  \draw [blue] (0.5,-0.1) -- (0.5,1.6);
  \draw [blue] (1,-0.1) -- (1,1.6);
\end{tikzpicture}
\hspace{2cm}
\begin{tikzpicture}[scale=2.5]
  \filldraw (0,0) circle (.5pt) node [below,left,black] {$0$};
  \filldraw (1,0) circle (.5pt) node [below,right,black] {$4K$};
  \filldraw (0,.75) circle (.5pt) node [below,left,black] {$2\I K'$};
  \filldraw (0,1.5) circle (.5pt);
  \filldraw (1,.75) circle (.5pt);
  \filldraw (1,1.5) circle (.5pt);
  \draw [blue] (0,-0.1) -- (0,1.6);
  \draw [blue] (0.5,-0.1) -- (0.5,1.6);
  \draw [blue] (1,-0.1) -- (1,1.6);
  \draw [dashed] (0.01,0) -- (1.01,0) -- (1.01,.75) -- (0.01,.75) -- (0.01,0);
\end{tikzpicture}
\caption{The lattice $\Lambda$ of periods in the unipartite case
  (left) and the bipartite case (right). The vertical lines project to
  the real locus on $\mathbb C/\Lambda$. The dashed line indicate a
  fundamental domain}\label{f-3}
\end{figure}

\begin{lemma}\label{l-2} Let $K$, $K'$ be defined for $0<k<1$
  to be real and positive. Extend the definition to the case of $k$ in
  the upper half plane by analytic continuation.
  \begin{enumerate}
  \item[I.] If $|D|<2$ then $k^2<0$. Let $k=\I \ell$ with $\ell>0$.
    Then $K$ is real and $2\I K'=-2K+2\I K''$ where
    \[
      K''=\int_{0}^{1/\ell}\frac{dv}{\sqrt{(1+v^2)(1-\ell^2v^2)}} \in
      \mathbb R_{>0}.
    \]
    Thus $\Lambda$ is spanned by $\omega=2K+2\I K''$,
    $\bar\omega=-2K+2\I K''$.
  \item[II.] If $|D|>2$ then $0<k^2<1$ and $\Lambda$ is spanned by
    $\omega_1=4K$ and $\omega_2=2\I K'$.
   \end{enumerate}
 \end{lemma}
 \begin{proof}
   The denominator of $k^2$ is positive by Lemma \ref{l-3}.  With the
   identity $(D+4E+2R)(D+4E-2R)=D^2-4$ we see that the numerator is
   positive if and only if $D^4>4$. In this case the numerator is
   smaller than the denominator since $R>0$.  Thus $k^2<0$ if $D^2<4$
   and $0<k^2<1$ if $D^2>4$.  Let $|D|<2$. Then $K$ is positive and
   $\I K'$ is given by analytic continuation for $k=\I\ell$ in the
   upper half-plane. The integration path from $1$ to $1/k=-\I/\ell$
   can be deformed to a path going from $1$ to $0$ and then continuing
   to $-i/\ell$ along the imaginary axis. Thus
   $\I K'=-K+\I\int_0^{1/\ell}dv/\sqrt{(1+v^2)(1-\ell^2v^2)}$.
   In the case II, no analytic continuation is needed.
 \end{proof}
\begin{figure}
\begin{tikzpicture}[scale=.7]
  \filldraw[fill=gray!20,draw=gray!20,domain=-1:-0.25,variable=\a]
  plot({\a,-1/\a-\a})-- (7,4.25)--(7,-4)--(4,-4)--(-2,2)--(-1,2);

  \draw[blue] (-2,2) -- (7,2) node [right] {$D=2$};

  \draw[blue] (2,-2) -- (7,-2) node [right] {$D=-2$};

  \draw [->] (0,-4) -- (0,4.5) node[right] {$D$};
  \draw [->] (-2.1,0) -- (5,0) node[below] {$E$};

  \draw (6,3.5) node {II${}_+$};
  \draw (6,0) node {I};
  \draw (6,-3) node {II${}_-$};
  \foreach \x in {-2, 2, 4}   \draw (\x ,.1) -- (\x ,-.1);
  \draw (2,0) node [below] {$1$};
  
  \foreach \y in {-3, -2, -1,1,2,3,4}   \draw (.1,\y) -- (-.1,\y);

\end{tikzpicture}
\caption{The physical domain of parameters, with the three
  regions delimited by the singular locus $D=\pm 2$. In the region I
  the real locus of the elliptic curve
  is isomorphic to $S^1$. In the regions II${}_\pm$
  it is isomorphic to $S^1\times\mathbb Z/2
  \mathbb Z$ and the dynamics is given by the translation of an
  element in the connected component of the identity  (case II${}_-$) or
  the other component (case II${}_+$).}\label{f-4}
\end{figure}
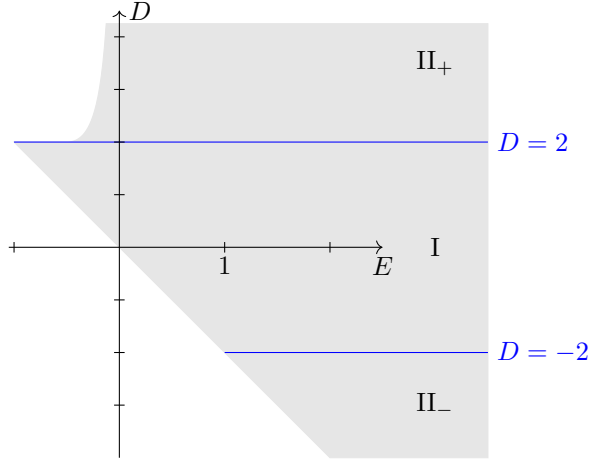

\begin{theorem}\label{t-6} Assume
  $(D,E)\in\mathbb R^2$ obeys the hypotheses
  of Theorem \ref{t-2} and the real angular momentum condition $D+2E>0$.
  Then $X_{\mathbb R}(D,E)$ is non-empty in the following three cases
  (see Fig.~\ref{f-4}). 
  \begin{enumerate}
  \item[I\quad] Let $-2<D<2$, $D+2E>0$.  In this case $k^2<0$,
    $|s_0|>1$. Write $k=\I \ell$.  Then there is a
    diffeomorphism
    $\varphi\colon \mathbb R/\mathbb Z\to X_{\mathbb R}(D,E)$ such
    that $\varphi^{-1}\circ t\circ \varphi(\theta)=\theta+\alpha$ is
    the shift by
    \[
      \alpha=\alpha(D,E)=-
      \frac{1}{4K''}\int_{-1}^{1/s_0}\frac{ds}{\sqrt{(1-s^2)(\ell^2+s^2)}}.
    \]
  \item[II${}_+$] Let $D>2$, $E>-1/2$, $1+2ED+4E^2>0$.  In this case
    $0<k^2<1$ and we have $1<s_0<1/k$ where $k$ is the positive square
    root.  Then there is a diffeomorphism
    $\varphi\colon \mathbb R/\mathbb Z\times\mathbb Z/2\mathbb Z\to
    X_{\mathbb R}(D,E)$ such that
    $\varphi^{-1}\circ t\circ
    \varphi(\theta,\epsilon)=(\theta+\alpha,\epsilon+\bar 1)$ is the
    composition of the shift by
    \[
      \alpha=\alpha(D,E)=\frac{1}{2K'}\int_{1}^{s_0}
      \frac{ds}{\sqrt{(s^2-1)(1-k^2s^2)}}
    \]
    and the generator $\bar 1$ of the subgroup $\mathbb Z/2\mathbb Z$.
  \item[II${}_-$] Let $D<-2$, $D+2E>0$. In this case $0<k<1$ and
    $-1/k<s_0<-1$.  There is a diffeomorphism
    $\varphi\colon \mathbb R/\mathbb Z\times\mathbb Z/2\mathbb Z\to
    X_{\mathbb R}(D,E)$ such that
    $\varphi^{-1}\circ t\circ
    \varphi(\theta,\epsilon)=(\theta+\alpha,\epsilon)$ is the shift by
    \[
      \alpha=\alpha(D,E)=-\frac{1}{2K'}\int_{s_0}^{-1}
      \frac{ds}{\sqrt{(s^2-1)(1-k^2s^2)}}.
    \]
  \end{enumerate}
  In these formulas we take the positive square root in the integrands.
  In all cases $C^2=(D+2E)(D+4E+2R)>0$.
  The map $\varphi$ is the parametrization of Theorem \ref{t-4} with
  $C>0$ and 
  with $u=4\I K''\theta$ in case $I$ and $u=2\I K'\theta+\epsilon 2K$
  in case II${}_\pm$.
  Moreover for any $E$, the derivative  $\partial\alpha(D,E)/\partial D$
  does not vanish on a dense open set.
\end{theorem}
\begin{proof}
  By Lemma \ref{l-3}, $X_{\mathbb R}(D,E)$ is non-empty if and only if
  \begin{equation}\label{e-6}
    D+4E+2R>0.
  \end{equation}
  This condition is empty if $E\geq 0$ since the left-hand side is the
  sum of $D+2E>0$, $2E\geq0$, $2R>0$.
  So let $E<0$.
  Eliminating $D$ yields the condition
  \[
    \frac{(R+2E+1)(R+2E-1)}{2E}\geq0.
  \]
  Since $R^2=1+(D+2E)2E<1$, the second factor in the numerator is negative,
  so the condition reduces to $R+2E+1\geq0$. This condition is empty
  if $E\geq-1/2$. If $E<-1/2$, it can be written as $R^2> (2E+1)^2$,
  which is equivalent to $2ED > 4E$, i.e.,$D<2$.
  We conclude
  that \eqref{e-6} is automatically satisfied if $D<2$ (under the assumption
  that $D+2E>0$), and requires $E>-1/2$ if $D>2$. In all cases,
  $C^2=(D+2E)(D+4E+2R)>0$.

  Similar elementary considerations lead to the following inequalities
  for $s_0$.
  \begin{enumerate}
  \item[I.] Let $-2<D<2$, $E>-D/2$. In this case $k^2<0$,
    $|s_0|>1$. Write $k=\I \ell$. The automorphism $t$ acts on
    $\mathbb C/\Lambda$ by translation by the elliptic integral $T$ of
    Theorem \ref{t-4}. Since we take $u=4\I K''\theta$ the shift
    of $\theta$ is $T/(4\I K'')$.
    The only subtlety is to fix the sign.
    The sign of the square root in the integrand is determined by
    the condition that its value at $s=s_0$ is $y(Q_+)$ which is
    negative imaginary (see Remark \ref{r-1}).  Thus we can write $T$
    as
    \[
      T=\I \int_{-1}^{s_0}\frac{ds}{\sqrt{(s^2-1)(1+\ell^2s^2)}}.
    \]
    where we integrate over a path connecting $-1$ and $s_0$ along the
    negative real axis (and going through the point at infinity if
    $s_0$ is positive). The square root is taken to be positive.  A
    more convenient expression is obtained by the variable
    substitution $s\mapsto 1/s$ since $1/s_0$ is then in the interval
    $(-1,1)$ and the integral is over a finite interval.
\item[II${}_+$.] Let $D>2$, $1+2 DE+4E^2>0$. In this case
  we can take $0<k<1$ and $1<s_0<1/k$ The integrand in the
  elliptic integral $T$ is real in the interval $(-1,1)$ and imaginary
  in the interval $[1,s_0]$. The integral over $(-1,1)$ is $2K$ (the
  sign does not matter here since $-2K\equiv 2K\mod\Lambda$) and
  the sign of the integral from $1$ to $s_0$ is fixed as above by
  the condition that the value of the square root at the end point $s_0$
  is negative imaginary. We get
  \[
    T=2K+\I \int_{1}^{s_0}\frac{ds}{\sqrt{(s^2-1)(1-k^2 s^2)}}.
  \]
\item[II${}_-$.] Let $D<-2$, $E>-D/2$. In this case $0<k<1$ and
  $-1/k<s_0<-1$. This case is treated similarly to the previous one.
\end{enumerate}
Finally, the fact that $\alpha(D,E)$ has non-zero derivative follows
by looking at the limit $D\to -2E$.  In this limit $s_0$ tends to
$-1$, so $\alpha(D,E)$, which does not vanish being given by an
integral of a positive function, tends to $0$. The paramter $k^2$ of
the elliptic curve tends to $(E-1)/(E+1)$ which is never 1, so that
$K''$ and $K'$ have a finite non-zero limit. Thus
$D\mapsto\alpha(D,E)$ is a real analytic function which is
non-constant on some interval. It has therefore only isolated critical
points.
\end{proof}
This parametrization allows us to plot the level sets and the orbits. We do
this in Fig.~\ref{f-5} using the Laplace--Runge--Lenz vector and the angular momentum
$L=z(D+2E)^{-\frac12}$ as coordinates.
\begin{figure} 
  \includegraphics[width=.4\textwidth]{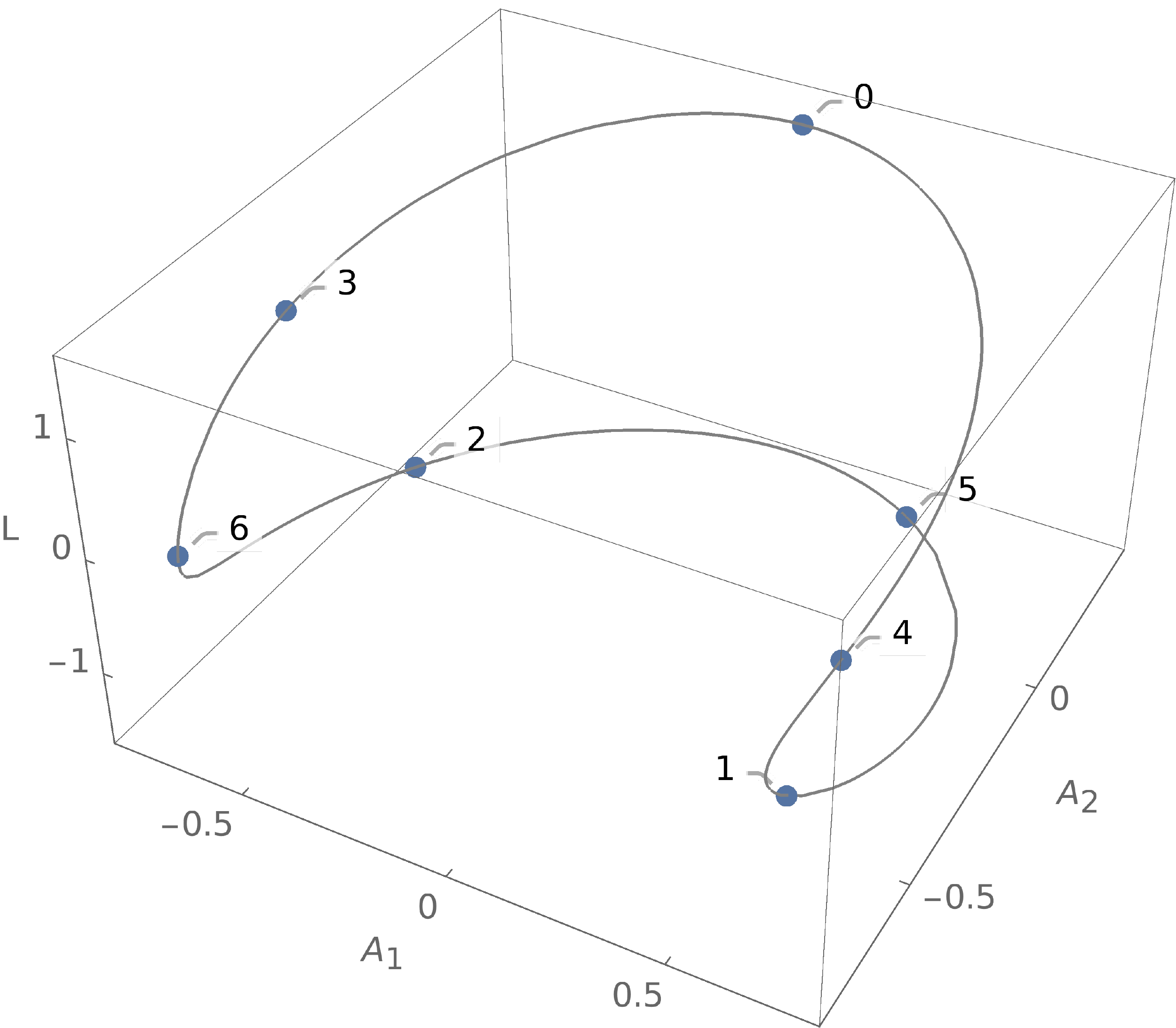}
  \includegraphics[width=.4\textwidth]{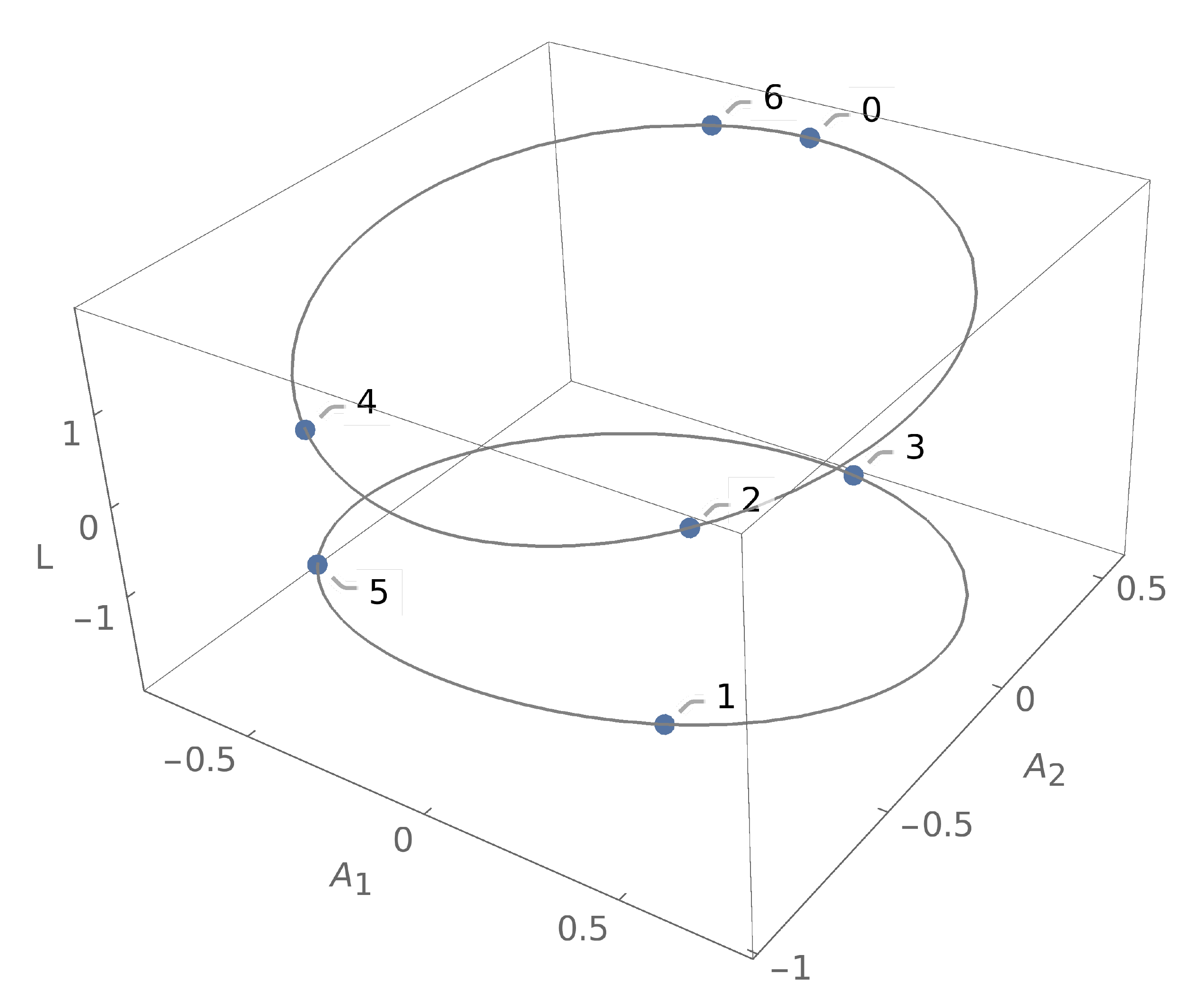}
  \caption{The level sets $X_{\mathbb R}(1.5,-0.2)$ (Type I, left) and
    $X_{\mathbb R}(2.5,-0.1)$ (Type II${}_+$, right) in the
    coordinates $A_1,A_2$ of the Laplace--Runge--Lenz vector and the angular
    momentum $L$, with six iterations of $t$.}
  \label{f-5}
\end{figure}

\begin{remark}
  In the negative energy case $E<0$ the motion takes place in a
  compact region of the plane so
  $X_{\mathbb R}(D,E)=\bar X_{\mathbb R}(D,E)$.  If $E>0$ one has
  collision points at infinity corresponding to hyperbolae with an asymptote
  parallel to the wall.
\end{remark}

\subsection*{Acknowledgments}
This research was inspired by a talk of Giovanni Gallavotti at ETH
Zurich in June 2019. I thank him for his talk and explanations, and
for providing me with a first draft of \cite{Gallavotti_Jauslin_2020}.
I also thank J\"urg Fr\"ohlich and Gian Michele Graf for their useful
comments. I am also grateful to Alexander P. Veselov who introduced me
to the Poncelet theorem and its relation with elliptic curves.

This paper is dedicated to the memory of Boris Dubrovin, whose unique
insights connecting integrable systems and algebraic geometry with a
deep root in classical analysis and geometry were always a great
inspiration for the author.

Figures 1,2 and 5 were produced with Mathematica.




  \bibliographystyle{plain}
  \bibliography{Boltzmann_Bibliography} 

\end{document}